\documentclass{article}[11pt]
\usepackage{graphicx}
\usepackage{psfrag}
\usepackage{epsf}
\usepackage{amsmath,amsfonts,amssymb,latexsym}
\usepackage{enumitem}
\usepackage{algorithmic}
\usepackage{algorithm}
\usepackage[width=145mm,height=205mm,left=35mm,foot=10mm]{geometry}
\usepackage{setspace}
\usepackage[cp1250]{inputenc}

\newcommand{\ignore}[1]{}
\newcommand{\startClaims}{\setcounter{claim}{0}}
\newtheorem{theorem}{Theorem}
\newtheorem{corollary}[theorem]{Corollary}
\newtheorem{lemma}[theorem]{Lemma}

\newtheorem{problem}{Problem}

\newtheorem{definition}[theorem]{Definition}

\title{A generalization of Hungarian method and Hall's theorem with applications in wireless sensor networks}

     \author{
     \textsc{Drago Bokal}\thanks{Supported in part by the Ministry of Science of Slovenia under the grant P1-0297} \\[0.25em]
     {\small{Faculty of Natural Sciences and Mathematics}} \\[-0.25em]
     {\small{University of Maribor}} \\[-0.25em]
     {\small{Slovenia}} \\[-0.1em]
     {\small\texttt{drago.bokal@uni-mb.si}}
     \and
     \textsc{Bo\v stjan Bre\v sar}$^{\ast}$ \\[0.25em]
     {\small{Faculty of Natural Sciences and Mathematics}} \\[-0.25em]
     {\small{University of Maribor}} \\[-0.25em]
     {\small{Slovenia}} \\[-0.1em]
     {\small\texttt{bostjan.bresar@uni-mb.si}}
     \and
     \textsc{Janja Jerebic}$^{\ast}$ \\[0.25em]
     {\small{Faculty of Natural Sciences and Mathematics}} \\[-0.25em]
     {\small{University of Maribor}} \\[-0.25em]
     {\small{Slovenia}} \\[-0.1em]
     {\small\texttt{janja.jerebic@uni-mb.si}}
     }

\newcommand{\DEF}[1]{{\em #1\/}}

\newenvironment{proof}%
{\noindent{\bf Proof.}\ }%
{\hfill\eopf\par\bigskip}%

{\noindent{\bf Sketch of a proof.}\ }%
{\hfill\eopf\par\bigskip}%
{\noindent{\bf Ideas for a proof.}\ }%
{\hfill\eopf\par\bigskip}%
{\noindent{\bf Proof in progress.}\ }%
{\hfill\eopf\par\bigskip}%

\newcommand{\dfnc}[3]{#1:#2\rightarrow #3}
\newcommand{\dset}[2]{\left\{#1 \:|\: #2\right\}}
\newcommand{\lset}[2]{\left\{#1, \ldots, #2\right\}}

\newcommand{\NN}{\mathbb N}

\newcommand{\RR}{\mathbb R}

\def\i4c{{internally-4-connected}}
\def\2cc{{2-crossing-critical}}

\def\m2{{{\cal M}_2}}

\newcommand{\eopf}{\raisebox{0.8ex}{\framebox{}}}

\newenvironment{wlist}
{\vspace{-10pt}
\begin{list}{}
{\setlength{\labelwidth}{15mm}
\setlength{\partopsep}{0pt}
\setlength{\parskip}{0pt}
\setlength{\topsep}{0pt}
\setlength{\itemsep}{0pt}
\setlength{\parsep}{0pt}
\setlength{\labelsep}{10pt}
\setlength{\leftmargin}{15mm}}
\item[]}
{\end{list}
\vspace{2pt}
\smallskip}

\begin{document}

\maketitle

\begin{abstract}
In this paper, we consider various problems concerning
quasi-matchings and semi-matchings in bipartite graphs, which
generalize the classical problem of determining a perfect matching
in bipartite graphs. We prove a vast generalization of Hall's
marriage theorem, and present an algorithm that solves the problem
of determining a lexicographically minimum $g$-quasi-matching (that
is a set $F$ of edges in a bipartite graph such that in one set of
the bipartition every vertex $v$ has at least $g(v)$ incident edges
from $F$, where $g$ is a so-called need mapping, while on the other
side of the bipartition the distribution of degrees with respect to
$F$ is lexicographically minimum). We also present an application in
designing an optimal CDMA-based wireless sensor networks.
\end{abstract}

\bigskip

\noindent {\bf Keywords:} matching, quasi-matching, 
semi-matching, flow, Hungarian method, augmenting path \\
\noindent {\bf AMS subject classification (2000)}: 05C70, 68R10,
05C90

\section{Introduction}

Problems related to matchings and factors belong to the classical
and intensively studied problems in graph theory. We refer to the
monograph of Lov\' asz and Plummer \cite{lp-86} from over 20 years
ago which is still one of the most comprehensive surveys on the
topic. Since the seminal paper of P. Hall \cite{h-35} containing a
characterization of perfect matchings in bipartite graphs, many
generalizations and variations of matchings and factors in
(bipartite) graphs have been considered. Let us mention the concepts
of 2-matchings, weighted matchings and $f$-factors \cite{lp-86}. At
least as much interest has been given to algorithmic issues
related to matchings, where a similarly influential role is played
by the famous max-flow min-cut theorem of Ford and Fulkerson
\cite{fofu-62}, cf. \cite{lp-86}. The research in the area is still
vivid, which is in part due to its applicability. Notably
applications often require special properties and yield different
variants of existing concepts which were not previously covered by
the theory. In this paper, we introduce and study the so-called
$f,g$-quasi-matching as a natural generalization of matchings in
bipartite graphs.

When modeling CDMA-based wireless sensor networks with graphs
\cite{ben,lbpj}, the following routing problem was encountered
(naturally, it can appear in any communication network with similar
features). The topology of the network is given by the nodes (in our
case sensor units) that are able to communicate among each other
with respect to physical limitations and their mutual distance.
There is a special vertex, the sink, represented by a fixed station
with relatively large computational capabilities. In our model, we
assume that nodes are also fixed, and they can also communicate with
the sink, depending on the mentioned limitations. This yields the
initial rooted graph, in which we wish to pass information from
nodes to the root. While nearby nodes communicate directly with the
sink, other (remote) nodes can pass information to the sink by using
other nodes as communication devices. For the purpose of energy
saving and latency, the number of hops from a given node to the
station must be as small as possible. The overall aim is to design a
routing protocol, by which each node in the network transfers
information to the sink as quickly as possible. Translating our
problem to graphs, we wish to find a spanning tree in a given rooted
graph using only edges that connect two different distance-levels
with respect to the root. There are many such trees obtainable by an
ordinary BFS-algorithm, yet they may have vertices with relatively
large degree, which can cause both communication delay and large
energy consumption of these nodes. Since the life-time of the
network depends on its weakest nodes, such situations need to be
avoided. See \cite{ak-04} for more on wireless sensor networks and
their routing protocols. We remark that finding a spanning tree with
the smallest maximum degree in a non-rooted graphs is a rather well
studied problem (see \cite{goemans} and the references therein), yet
it does not have much connection with the problem on rooted graphs.

Our situation can be quickly translated to the following
optimization problem. {\it Given a rooted graph, find a
spanning tree with maximum degree as small as possible.} Another
more general problem follows from the requirement that more than one
path from a node to the sink is needed, either to provide robustness
against possible node failures or to avoid communication delay due to
collisions at more frequent nodes. Hence alternative paths need
to be determined in advance. Then the problem is to {\it find a
spanning subgraph with maximum degree as small as possible in which
each vertex has $k$ neighbors in the neighboring distance level that
is closer to the root}. More generally, if we have a traffic
estimation at the nodes, then the number of neighbors in the lower
level can be assigned to each vertex individually. By concentrating
solely on two neighboring levels, the problem is to find a spanning
subgraph in a bipartite graph such that, in one set of the partition,
the degrees of vertices are prescribed: they can be 1 (derived from
the  original problem), have a fixed degree $k$ (for the so-called
multipath routing), or they can be determined by an arbitrary
function that corresponds to estimated traffic at the nodes. In
the other set of the bipartition, we are either aiming at the
minimization of the largest degree (optimization problem), or we are
also facing some constraints on degrees of vertices (decision
problem). We will address both of these problems.

A variation of the first (and the simplest) of the mentioned
problems was considered in \cite{hllt}, with motivation arising from
some task scheduling. The authors introduced the so-called
semi-matchings which coincide with spanning forests in bipartite
graphs and their objective was the reduction of a certain
cost-function that is connected to the maximum degree of a forest.
We present a solution to the more general problem of determining an
optimal quasi-matchings, where on one side of the bipartition
degrees of vertices with respect to a quasi-matching obey specified
lower bounds, while on the other side not only the maximum degree of
vertices is minimized, but also their degree distribution is
lexicographically minimum. As it turns out, the resulting algorithm
is on-line, in the sense that an increase or decrease of a lower
bound by one in a vertex, after the semi-matching has been built,
requires only one additional step to obtain an optimal semi-matching
of the graph with new bounds.

In the next section, we fix the notation and present the main
problems, expressed in the language of graph theory. In Section
\ref{sc:hungarian}, the Hungarian method is extended to the above
mentioned problem of finding a lexicographically minimum
quasi-matching in a bipartite graph that yields an efficient
algorithm for the original problem. This algorithm is presented as
an off-line algorithm, although it can be interpreted as an on-line
algorithm when only additions of vertices or the increase of the
prescribed lower bounds occur. It is extended in Section
\ref{sc:online} to the case when the prescribed lower bound
decreases (or the vertex is deleted). In Section \ref{sc:hall}, we
consider a decision version of the most general problem that comes
from the above discussion. We prove a characterization of bipartite
graphs that admit a spanning subgraph in which for the degrees of
vertices of one of the sets in the partition arbitrary lower bounds
are imposed, while in the other set of the partition degrees of
vertices with respect to the spanning subgraph need to obey
arbitrarily specified upper bounds. This result is a vast
generalization of the famous Hall's marriage theorem.

\section{Quasi-matchings in bipartite graphs}

This section introduces the terminology used throughout the paper.
We also characterize minimum semi-matchings and establish their
various properties concerning optimality.

\begin{definition}
Let $G=A+B$ be a bipartite graph. Given a positive integer $k$, a
set $F\subseteq E(G)$ is a \DEF{$k$-quasi-matching} of $Y\subseteq
B$, if every element of $Y$ has at least $k$ incident edges from
$F$. A 1-quasi-matching of $Y$ in which every element of $Y$ has
exactly $1$ incident edge from $F$ is called a {\em semi-matching}.
\end{definition}

\begin{definition}
Let $G=A+B$ be a bipartite graph and $g\colon B\rightarrow \NN$ a
mapping. For a vertex $v\in B$ we call $g(v)$ the \DEF{need} of $v$,
and for any $Y\subseteq B$, the \DEF{need} of $Y$ is
$g(Y)=\sum_{v\in Y}{g(v)}$. A set $F\subseteq E(G)$ is a
\DEF{$g$-quasi-matching} of $Y\subseteq B$ if  every element $v$ of
$Y$ has at least $g(v)$ incident edges from $F$. Next, for a mapping
$f\colon A\rightarrow \NN$, and a vertex $u\in A$ we call $f(u)$ the
\DEF{capacity} of $u$, and for any $X\subseteq A$, the
\DEF{capacity} of $X$ is $f(X)=\sum_{u\in X}{f(u)}$. A set
$F\subseteq E(G)$ is an \DEF{$f,g$-quasi-matching} of $A+B$ if every
element $v$ of $Y$ has at least $g(v)$ incident edges from $F$, and
every element $u$ of $X$ has at most $f(u)$ incident edges from $F$.
\end{definition}

Note that a $g$-quasi-matching of $B$ with a constant need function,
$g(v)=k$, for all $v\in B$, is a $k$-quasi-matching of $B$.

\begin{definition}
Let $G=A+B$ be a bipartite graph and $F\subseteq E(G)$. For a vertex
$v\in V(G)$, the $F$-degree of $v$, $d_F(v)$ is the degree of $v$ in
$G[F]$. The \DEF{degree} of $F$ is the maximum degree in $G[F]$ of a
vertex from $A$.
\end{definition}

Note that a matching of $Y\subseteq B$ is a semi-matching of $Y$ with
degree equal to $1$. We are interested in the following two problems.

\begin{problem}
\label{pb:minDeg} Given a bipartite graph $G=A+B$ and a need
function $g$ on $B$, find a $g$-quasi-matching of $B$ with minimum
degree.
\end{problem}

\begin{problem}
\label{pb:minDegGen} Given a bipartite graph $G=A+B$, is there an
$f,g$-quasi-matching of $A+B$?
\end{problem}

We solve the first problem by generalizing Hungarian method in
Section 3 and the second one by giving a characterization that
generalizes Hall's theorem in Section \ref{sc:hall}.

\begin{definition}
Let $G=A+B$ be a bipartite graph, let $F\subseteq E(G)$, and
$X\subseteq A$. Let $d_F(X)$ be the sequence $d_1,d_2,\ldots,
d_{|X|}$ of $F$-degrees of vertices from $X$, where $d_1\geq d_2
\geq \cdots \geq d_{|X|}$. For $Y\subseteq B$, we define
$d_F(Y)=d_F(N(Y))$.
\end{definition}

The following definition applies to all types of quasi-matchings
(integer, $g$-quasi-matchings and $f,g$-quasi-matchings).

\begin{definition}
Let $G=A+B$ be a bipartite graph, let $F,F'$ be two quasi-matchings
of $Y\subseteq B$. Then $F$ is (lexicographically) greater than
$F'$, if $d_F(Y)$ is lexicographically greater than $d_{F'}(Y)$. A
quasi-matching $F$ of $Y\subseteq B$ that is not greater than any
other quasi-matching of $Y$ is a \DEF{minimum} quasi-matching of
$Y$.
\end{definition}

Clearly, a minimum quasi-matching of $B$ has a minimum degree. It is
also easy to see that in a minimum $g$-quasi-matching all vertices
in $B$ have $F$-degree equal to their need. Thus, to solve Problem
\ref{pb:minDeg}, we propose

\begin{problem}
\label{pb:minimum} Given a bipartite graph $G=A+B$ and a need
function $g\colon B\rightarrow \NN$, find a (lexicographically)
minimum $g$-quasi-matching of $B$.
\end{problem}

An on-line algorithm for solving Problem \ref{pb:minimum} is one of
the major contributions of this paper. We start with the following
easy lemma. (Recall that the {\em pigeonhole or Dirichlet principle}
states that given a set of $t$ objects that are placed into boxes,
and there are $s$ boxes available, then there will be a box
containing at least $\lceil \frac {t}{s}\rceil$ objects.)

\begin{lemma}
\label{Lgolob} Let $G=A+B$ be a bipartite graph, $g\colon
B\rightarrow \NN$ a need function, and $F$ a $g$-quasi-matching of
$B$. Let $X\subseteq A$, with $|X|=k$, and let $Y=N(X)$ be the set
of their neighbors. Let $t$ be the number of edges with one
end-vertex from $Y$ and the other from $A-X$, and let $g(Y)=t+dk+r$,
where $0\le r<k$ and $d\geq 0$. Then $d_F(X)$ is lexicographically
greater or equal to the distribution with $r$ integers $d+1$ and
$k-r$ integers $d$.
\end{lemma}

\begin{proof}
Note that $d_F(X)$ is (lexicographically) the smallest only if all
edges with one end-vertex from $Y$ and the the other from $A-X$ are
in $F$. We may thus assume without loss of generality that this is
the case. Hence $\sum_{x\in X} d_F(x)=dk+r$.

If $r=0$, then either $d_F(X)$ consists of precisely $k$ integers
$d$ or $d_F(x)$ contains at least one integer strictly greater than
$d$. Both distributions are lexicographically greater or equal to
the distribution with $k$ integers $d$.

So suppose $r>0$ implying $k>r\ge 1$. By applying Dirichlet's
principle, either $X$ contains a vertex $a$ with $d_F(a)> d+1\ge 1$
(in which case $d_F(X)$ is lexicographically greater than the
distribution with the largest degree $d+1$) or there are $r$
vertices in $X$ with $F$-degree $d+1$ and $k-r$ vertices in $X$ with
$F$-degree $d$. The claim follows.
\end{proof}


\begin{definition}
Let $G=A+B$ be a bipartite graph and $F\subseteq E(G)$ a set of
edges. A \DEF{(forward) $F$-alternating path} from a vertex $a\in A$
to a vertex $a'\in A$ in $G$ is a path $P$ such that every internal
vertex of $P$ is in $P$ incident with one edge in $F$ and another
not in $F$, and that $a$ is in $P$ incident with $F$, but $a'$ is
not. A path $P$ from a vertex $a\in A$ to a vertex $a'\in A$ in $G$
is a \DEF{backward $F$-alternating path} if the reversed path on the
same edges from $a'$ to $a$ is a (forward) $F$-alternating path. An
\DEF{$F$-augmenting path} $P$ in $G$ is a path from a vertex $b\in
B$ to a vertex $a\in A$, such that $P-b$ is an $F$-alternating path
from $a'$ to $a$, and the edge $a'b$ is not in $F$.
\end{definition}

Note that by performing $F$-exchange $F'=F\oplus E(P)$ of edges in
an $F$-alternating path $P$ from $a\in A$ to $a'\in A$, the degree
of $a$ decreases by one ($d_{F'}(a)=d_F(a)-1$), the degree of $a'$
increases by one ($d_{F'}(a')=d_F(a)+1$), and all other
quasi-matching-degrees remain as in $F$.

\begin{definition}
Let $G=A+B$ be a bipartite graph, $F$ a quasi-matching of
$Y\subseteq B$ and $P$ an $F$-alternating path from $a\in A$ to
$a'\in A$. The \DEF{decline} of $P$ is $dc(P)=d_F(a)-d_F(a')$.
\end{definition}

\begin{definition}
Let $G=A+B$ be a bipartite graph, $F \subseteq E(G)$, and $a\in A$.
The \DEF{$a$-section} of $G$ is a maximal subgraph $G_a=X_a+Y_a\subseteq G$,
such that there is an $F$-alternating path $P_{a'}$ from $a$ to every
$a'\in X_a$ and $Y_a=N_F(X_a)$ is the set of $F$-neighbors of $X_a$.
Furthermore, $F_a$ is the set of edges in $F$ incident with $X_a$.
\end{definition}

Thus defined $a$-sections play a crucial role in our proof of the
following characterization of minimum $g$-quasi-matchings.


\begin{theorem}
\label{th:declinenov} Let $G=A+B$ be a bipartite graph, $g\colon
B\rightarrow \NN$ a need function and $F$ a $g$-quasi-matching of
$B$. Then $F$ is a minimum $g$-quasi-matching of $B$ if and only if
any $F$-alternating path in $G$ has decline at most $1$.
\end{theorem}
\begin{proof}
Suppose there is an $F$-alternating path $P$ in $G$ whose decline is
at least two. By performing an $F$-exchange of edges on $P$, we get
a $g$-quasi-matching $F'$, such that $F$ is lexicographically
greater than $F'$, a contradiction.

The converse is by induction on $g(B)=\sum_{y\in B}g(y)$. Assume
that all $F$-alternating paths in $G$ have decline at most 1. Let
$a\in A$ be a vertex with the largest $F$-degree in $G$, and let
$H=X+Y$ be the $a$-section in $G$. Note that for any $a'\in X$,
$$d_F(a)-1 \leq d_F(a')\leq d_F(a).$$

Also note that by definition of the $a$-section (maximality), any
edge connecting a vertex from $Y$ to a vertex from $A-X$ is in $F$.
Let $t$ be the number of edges connecting a vertex from $Y$ to a
vertex from $A-X$. Then by letting $|X|=k$ and $d=d_F(a)$, we easily
infer that $g(Y)=t+k(d-1)+r$, where $r$ is the number of vertices in
$X$ with $F$-degree equal to $d$. By Lemma \ref{Lgolob}, the
distribution $d_F(X)$ coincides with the lexicographically minimum
degree distribution of a $g$-quasi-matching. Hence, if $X=A$ (and so
$t=0$), the proof is complete.

Thus, suppose that $X\neq A$. Let $Y'=N(A-X)$, and note that $Y\cup
Y'=B$, while $Y\cap Y'$ may be nonempty. Let $F'$ be the restriction
of $F$ to the edges with one endvertex in $A-X$, and set $F''=F-F'$
(i.e. $F''$ contains edges from $F$ that have one endvertex in $X$).
We set a need mapping $g'$ of $Y'$ with $g'(v)=g(v)-d_{F''}(v)$ for
any $v\in Y'$. Now, any $F'$-alternating path in $(X-A)+Y'$ has
decline at most one because $F'$ is just the restriction of $F$. As
$g'(Y')<g(B)$ we infer by induction hypothesis that $F'$ is a
(lexicographically) minimum $g'$-quasi-matching of $Y'$.

Let $Q$ be a minimum $g$-quasi-matching. Hence $d_Q(A)$ is not
greater than $d_F(A)$. In addition we infer by Lemma \ref{Lgolob}
that the distribution $d_Q(X)$ is at least $d_F(X)$, that is, there
is at least $r$ vertices from $X$ whose $Q$-degree is $d$. Let $p$,
$p\geq r$ be the number of vertices in $X$ whose $Q$-degree is $d$.
Denote by $Q''$ the set of edges from $Q$ that have one endvertex in
$X$, and let $Q'=Q-Q''$. Now we introduce a need mapping $g''$ on
$Y'$ by setting $g''(v)=g(v)-d_{Q''}(v)$ for any $v\in Y'$.  Note
that $g''(Y')=g'(Y')-(p-r)$, and so

\begin{equation}
\label{enacba}
 \sum_{u\in A-X}{d_{Q'}(u)}=\sum_{u\in A-X}{d_{F'}(u)}-(p-r).
\end{equation}

Note also that $g'(u)\geq g''(u)$ for any $u\in Y'$. Since $Q'$ is
clearly a minimum $g''$-quasi-matching of $(A-X)+Y'$ we infer (again
by induction hypothesis) that it has no alternating paths with
decline more than 1.

We gradually increase the $g''$-quasi-matching $Q'$ of $(A-X)+Y'$ to
a $g'$-quasi-matching by using the following procedure that consists
of $p-r$ steps. We denote by $Q_i$ the quasi-matching in the $i$-th
step of the procedure (and set $Q_0=Q'$). In each step we obtain
$Q_i$ from $Q_{i-1}$ by taking a vertex $u\in Y'$ with
$g''(u)<g'(u)$, for which $d_{Q_{i-1}}(u)<g'(u)$. Let $P$ be an
augmenting path from $u$ to a vertex $a_i$ of smallest possible
$Q_{i-1}$-degree in $A-X$. Then we set $Q_i=Q_{i-1}\oplus E(P)$.
Note that all vertices from $A-X$ on $P$ have degree $d_{Q_i}(a_i)$
because $P-u$ is a forward $Q_{i-1}$-alternating path, having
decline exactly 1 (unless $a_i$ is already a neighbor of $u$). From
this we quickly infer that there are no $Q_i$-alternating path with
decline more than 1, provided there were no such
$Q_{i-1}$-alternating paths. In the last step we get a
$g'$-quasi-matching $Q_{p-r}$ which thus has no alternating paths
with decline more than 1. By induction hypothesis $Q_{p-r}$ is a
minimum $g'$-quasi-matching of $(A-X)+Y'$ hence its degree
distribution in $A-X$ coincides with $d_{F'}(A-X)$.

From \eqref{enacba} we find that $d_{Q'}(A-X)$ is the smallest
possible (noting that it can be obtained from $d_{F'}(A-X)$ by
taking off $p-r$ units from vertex degrees in $A-X$) if there are
exactly $p-r$ vertices in $A-X$ with $Q'$-degree $d-1$ and whose
$F'$-degree is $d$ (in all other cases, the number of vertices with
$F'$-degree equal to $d$ is less than the sum of $p-r$ and the
number of vertices with $Q'$-degree equal to $d$, which would in
turn imply that $d_F(A)$ is strictly smaller than $d_Q(A)$). Now,
this implies that in other vertices of $A-X$ the distributions of
$d_{Q'}$ and $d_{F'}$ are the same. Combined with distributions of
degrees in $X$ we derive that $d_F(A)=d_Q(A)$, and so $F$ is a
minimum $g$-quasi-matching as well.
\end{proof}

The $1$-quasi-matchings alias semi-matchings were studied also in
\cite{hllt}. In order to connect our results to theirs, we adopt the
following definition.

\begin{definition}
Let $G=A+B$ be a bipartite graph, $F$ a semi-matching of $B$,
and $f:{\mathbb R}^+ \rightarrow {\mathbb R}$ a strictly (weakly)
convex function. Then the function
${\rm cost}_f$, defined as $\sum_{i=1}^{|A|}f(d_F(a_i))$ is called
a strict (weak) cost function for $f$.
\end{definition}

In \cite{hllt}, the strictly convex function
$\ell(n)=\frac{1}{2}n(n+1)$ is emphasized. It is interesting in task
scheduling, as it measures total latency of uniform tasks on a
single machine. It is also proved that a semi-matching $F$ has
minimum ${\rm cost}_{\ell}(F)$ if and only if any $F$-alternating
path in $G=A+B$ has decline at most 1. By Theorem
\ref{th:declinenov}, $F$-alternating paths in $G$ have such property
if and only if $F$ is (lexicographically) minimium semi-matching of
$B$. The special case of Theorem \ref{th:declinenov} where the need
function is constant 1 combined with the results from \cite{hllt}
leads to the following equivalent characteristics of the
(lexicographically) minimium semi-matching.

\begin{corollary}
\label{th:lovasz}
Let $G=A+B$ be a bipartite graph, $F$ a semi-matching of $B$,
and $\dfnc{f}{\RR^+}{\RR}$ a strictly convex function.
Then the following are equivalent:
\begin{wlist}
\item[(i)]  $F$ is (lexicographically) minimium semi-matching of $B$.
\item[(ii)] Any $F$-alternating path in $G$ has decline at most 1.
\item[(iii)]$F$ has minimum ${\rm cost}_{\ell}(F)$ for $\ell(n)=\frac{1}{2}n(n+1)$.
\item[(iv)] $F$ has minimum ${\rm cost}_f(F)$.
\item[(v)]  $L_p$-norm, $1\leq p <\infty$, of the vector
            $X=(d_F(a_1),\ldots,d_F(a_{|A|}))$ is minimal.
\item[(vi)] The variance of the vector $X=(d_F(a_1),\ldots,d_F(a_{|A|}))$
            is minimal.
\end{wlist}
\end{corollary}

\begin{proof}
The equivalence $(i)\iff (ii)$ follows from the Theorem \ref{th:declinenov}.
Furthermore, $(ii)$ is equivalent to $(iii)$ (\cite{hllt}, Theorem 3.1)
and $(iv)$ (\cite{hllt}, Theorem 3.5). Finally, $(iii)$ is equivalent to
$(v)$ (\cite{hllt}, Theorem 3.9) and $(vi)$ (\cite{hllt}, Theorem 3.10).
\end{proof}

Every property of the above Theorem \ref{th:lovasz} implies that $F$
has minimum ${\rm cost}_f(F)$ for every weakly convex function $f$
(\cite{hllt}, Theorem 3.5) and that $L_{\infty}$-norm of the
vector $X=(d_F(a_1),\ldots,$ $d_F(a_{|A|}))$ is minimal (\cite{hllt},
Theorem 3.12). In both cases, the converse is not true.

\begin{corollary}
\label{th:matching}
Let $G=A+B$ be a bipartite graph and let $F$ be a (lexicographically)
minimium semi-matching of $B$. Then there exists a maximum matching
$M\subseteq F$ in $G$.
\end{corollary}

\begin{proof}
Follows directly from Theorem \ref{th:lovasz} and Theorem 3.7 from \cite{hllt}.
\end{proof}

The converse of Corollary \ref{th:matching} does not hold (see
\cite{hllt}).

\section{Generalized Hungarian method}
\label{sc:hungarian}

In this section, we solve Problem \ref{pb:minimum} with an algorithm
of complexity $O(g(B)|E(G)|)$. We use the fact that quasi-matchings are
a generalization of matchings: if we restrict ourselves to
quasi-matchings with degree one, our method is a generalization of
the Hungarian method of augmenting paths for finding maximum
matchings in bipartite graphs.\\

\noindent Let $B=\lset{b_1}{b_n}$ and $B_\ell=\lset{b_1}{b_\ell}$,
$\ell=1,\ldots,n$. Define a mapping $g_i\colon
B\rightarrow \NN$ with $g_0(b)=0$, for all $b\in B$, $\ell_1=1$, 
$\ell_i=\max\dset{j}{g_{i-1}(b_j)\neq 0}$ for $i>1$, and
$$g_i(b) = \left\{ \begin{array}{ll}
         g_{i-1}(b)+1; & b=b_{\ell_i}\hbox{ and } g_{i-1}(b)<g(b), \\
         1; & b=b_{\ell_i+1}\hbox{ and } g_{i-1}(b_{\ell_i})=g(b_{\ell_i}), \\
         g_{i-1}(b); & {\rm otherwise}
\end{array} \right.$$
\noindent
for every $1\leq i \leq g(B)$. Note that for simplicity we assume $g(b)>0$ 
for all $b\in B$. We propose to find a minimum
$g$-quasi-matching $F$ of $B$ using an iterative algorithm that gradually
extends an $g_i$-quasi-matching $F_i$ of $B_{\ell}$
using an $F_{i-1}$-augmenting path $P_{i-1}$ from $b_{\ell}$ to $a\in A$
with smallest $d_{F_{i-1}}(a)$. By induction, we argue that $F_i$ is a
minimum $g_i$-quasi-matching of $B_{\ell}$, thus the final $F_i$ is a
minimum $g$-quasi-matching of corresponding $B_\ell=B$.

\begin{algorithm}
\caption{Iterative construction of a minimum $g$-quasi-matching of $B$.}
\label{al:iterative}
\begin{algorithmic}
\PARAMETER{$G=A+B$} a bipartite graph with $B=\lset{b_1}{b_{n}}$.
\OUTPUT{$F$} a minimum $g$-quasi-matching of $B$.
\STATE Set $i=0$, $\ell=0$.
\STATE Set $F_i=\emptyset$, $B_{\ell}=\emptyset$, $G_{\ell}=\emptyset$.
\WHILE{$\ell \leq n$}
  \STATE $\ell=\ell +1$.
  \STATE Set $B_{\ell}=B_{\ell -1}\cup\{b_{\ell}\}$.
  \STATE Set $G_{\ell}=G[B_{\ell -1}\cup A]$.
  \STATE $c=0$.
    \WHILE{$c<g(b_{\ell})$}
    \STATE $i=i+1$, $c=c+1$.
    \STATE Set $P_{i-1}$ to be an $F_{i-1}$-augmenting path
           in $G_{\ell}$ from $b_{\ell}$ to  $a\in A$ with smallest
           possible degree $d_{F_{i-1}}(a)$.
    \STATE Set $F_i=F_{i-1}\oplus E(P_{i-1})$.
    \ENDWHILE
\ENDWHILE
\RETURN $F_i$.
\end{algorithmic}
\end{algorithm}
\begin{lemma}
\label{stopnje}
Let $G=A+B$ be a bipartite graph and $a\in A$. Using the notation of Algorithm \ref{al:iterative}, the following holds:
$$d_{F_i}(a) = \left\{ \begin{array}{ll}
         d_{F_{i-1}}(a)+1; & {\rm if\hspace{2mm}} a {\rm \hspace{2mm} is\hspace{2mm} the\hspace{2mm}
          A{\rm -endvertex}\hspace{2mm} {\rm of}\hspace{2mm}} P_{i-1}, \\
         d_{F_{i-1}}(a); & {\rm otherwise}\,.
\end{array} \right.$$
\end{lemma}

\begin{proof}
The Lemma is obviously true for every vertex $a\in A\setminus V(P_{i-1})$.
Since $F_i=F_{i-1}\oplus E(P_{i-1})$, $e\in F_{i-1}\cap E(P_{i-1})$ implies
that $e\notin F_i$. Similarly, for every $e\in E(P_{i-1})\setminus F_{i-1}$
we have $e\in F_i$. Therefore, the number of $F_i$-edges at an internal
$P_{i-1}$ vertex $a$ is the same as the number of $F_{i-1}$-edges at $a$.
However, if $a$ is the $A$-endvertex of $P_{i-1}$, then its only
$P_{i-1}$ incident edge is not in $F_{i-1}$ but is in $F_i$, so
$d_{F_i}(a)=d_{F_{i-1}}(a)+1$.
\end{proof}

\begin{theorem}
\label{th:minimalen_F_i}
Let $G=A+B$ be a bipartite graph. Using the notation of Algorithm
\ref{al:iterative}, $F_i$ is a minimum $g_i$-quasi-matching of $B_{\ell}$
in $G_{\ell}$ for $i=1,\ldots,n$.
\end{theorem}

\begin{proof}
For $i=1$, we have $\ell=1$, $B_1=\{b_1\}$ and $g_1(b_1)=1$. Let $a$
be any vertex from $N(b_1)$. Then $F_1=P_0=b_1a$ is a minimum
$g_1$-quasi-matching of $B_1$ in $G_1$.

Suppose now that $F_{i-1}$ is a minimum $g_{i-1}$-quasi-matching of
$B'=B_{\ell}$ (or $B'=B_{\ell-1}$) in $G'=G_{\ell}$ (or
$G'=G_{\ell-1}$). We claim that $F_i=F_{i-1}\oplus E(P_{i-1})$ is a
minimum $g_i$-quasi-matching of $B_\ell$ in $G_\ell$. If this is not
the case, then Theorem \ref{th:declinenov} yields an
$F_i$-alternating path $P$ in $G_\ell$ from $a'\in A$ to $a''\in A$
with decline $d_{F_i}(a')-d_{F_i}(a'')\geq 2$. Note that every
$F_i$-alternating subpath of $P$ from a vertex $a\in A$ leads to
$a''$ and every backward $F_i$-alternating subpath leads to $a'$.

Consider first the case for $E(P)\cap E(P_{i-1})=\emptyset$.
Then an edge $e$ of $P$ is in $F_{i-1}$ if and only if it is in
$F_i$. For the rest of the proof let $a$ denote the endvertex of $P_{i-1}$.
We distinguish three cases:\\

{\bf Case A1:} $a\notin\{a',a''\}$.\\
Lemma \ref{stopnje} implies that $d_{F_i}(a')=d_{F_{i-1}}(a')$ and
$d_{F_i}(a'')=d_{F_{i-1}}(a'')$. Thus, $P$ is an
$F_{i-1}$-alternating path from $a'$ to $a''$ in $G'$ with
decline at least 2. A contradiction to Theorem \ref{th:declinenov},
since $F_{i-1}$ is a minimum $g_{i-1}$-quasi-matching of $B'$ in $G'$.\\

{\bf Case A2:} $a=a'$.\\
Lemma \ref{stopnje} implies $d_{F_i}(a')=d_{F_{i-1}}(a')+1$
and $d_{F_i}(a'')=d_{F_{i-1}}(a'')$. Let $v$ be the common vertex
of the paths $P$ and $P_{i-1}$ closest to $b_{\ell}$ in $P_{i-1}$.
Then $Q=b_{\ell}P_{i-1}vPa''$ (resp. $Q=b_{\ell}Pa''$ for $v=b_{\ell}$)
is an $F_{i-1}$-augmenting path in $G_\ell$. Since $P_{i-1}$ in $G_{\ell}$
is chosen so that $d_{F_{i-1}}(a)$ is minimum, we have
\begin{center}
$d_{F_{i-1}}(a)=d_{F_{i-1}}(a')\leq d_{F_{i-1}}(a'')$\\
$d_{F_{i-1}}(a')-d_{F_{i-1}}(a'')\leq 0$.
\end{center}

This contradicts the assumption $d_{F_i}(a')-d_{F_i}(a'')\geq 2$, as
\begin{center}
$d_{F_{i-1}}(a')+1-d_{F_{i-1}}(a'')\geq 2$\\
$d_{F_{i-1}}(a')-d_{F_{i-1}}(a'')\geq 1.$
\end{center}

{\bf Case A3:} $a=a''$.\\
In this case, Lemma \ref{stopnje} implies that $P$ is an
$F_{i-1}$-alternating path from $a'$ to $a''$ in $G'$ with
$d_{F_i}(a')=d_{F_{i-1}}(a')$ and $d_{F_i}(a'')=d_{F_{i-1}}(a'')+1$.
The inequality
$d_{F_i}(a')-d_{F_i}(a'')\geq 2$ yields
\begin{center}
$d_{F_{i-1}}(a')-d_{F_{i-1}}(a'')-1\geq 2$\\
$d_{F_{i-1}}(a')-d_{F_{i-1}}(a'')\geq 3$.
\end{center}

Hence, $P$ is an $F_{i-1}$-alternating path in $G'$
with decline at least 3. But this is again impossible by Theorem
\ref{th:declinenov} and minimality of $F_{i-1}$.\\

It remains to examine the case $E(P)\cap E(P_{i-1})\neq\emptyset$.

{\bf Case B1:} $a\notin\{a',a''\}$.\\
Let $v$ be the common vertex of $P$ and $P_{i-1}$ closest
to $b_{\ell}$ in $P_{i-1}$. Then $Q=b_{\ell}P_{i-1}vPa''$
(resp. $Q=b_{\ell}Pa''$ for $v=b_{\ell}$) is an
$F_{i-1}$-augmenting path in $G_{\ell}$. The choice of $P_{i-1}$ implies
$d_{F_{i-1}}(a)\leq d_{F_{i-1}}(a'')$.

Let $v'$ be the common vertex of $P$ and $P_{i-1}$ closest to $a'$ in $P$.
Then $R=a'Pv'P_{i-1}a$ is an $F_{i-1}$-alternating
path in $G_{\ell}$. Since Lemma \ref{stopnje} implies
$$2\leq d_{F_i}(a')-d_{F_i}(a'')=d_{F_{i-1}}(a')-d_{F_{i-1}}(a'')\leq
d_{F_{i-1}}(a')-d_{F_{i-1}}(a),$$
$R$ is a path with $F_{i-1}$-decline at least two, another contradiction
to Theorem \ref{th:declinenov} and minimality of $F_{i-1}$.

{\bf Case B2:} $a=a'$.\\
Let $Q$ be the $F_{i-1}$-augmenting path in $G_{\ell}$ from $b_{\ell}$
to $a''$ as in case B1. The existence of such a path ensures that
$d_{F_{i-1}}(a)-d_{F_{i-1}}(a'')\leq 0$. But this is not possible,
since Lemma \ref{stopnje} implies
$$2\leq d_{F_i}(a')-d_{F_i}(a'')=d_{F_i}(a)-d_{F_i}(a'')=
d_{F_{i-1}}(a)+1-d_{F_{i-1}}(a'')$$
and hence $d_{F_{i-1}}(a)-d_{F_{i-1}}(a'')\geq 1$.

{\bf Case B3:} $a=a''$.\\
Let $R$ be the $F_{i-1}$-alternating path in $G_{\ell}$ from $a'$ to $a$
constructed as in case B1. We claim that $R$ has decline
at least three. From $d_{F_i}(a')-d_{F_i}(a'')\geq 2$ and Lemma
\ref{stopnje}, we deduce that
\begin{center}
$d_{F_{i-1}}(a')-d_{F_{i-1}}(a'')-1\geq 2$\\
$d_{F_{i-1}}(a')-d_{F_{i-1}}(a'')\geq 3$\\
$d_{F_{i-1}}(a')-d_{F_{i-1}}(a)\geq 3$.\\
\end{center}
But this contradicts the minimality of $F_{i-1}$.

We conclude that $F_i$ is a minimum $g_i$-quasi-matching
of $B_{\ell}$ in $G_{\ell}$.
\end{proof}

By setting $\ell=n$, Theorem \ref{th:minimalen_F_i} proves
correctness of the Algorithm \ref{al:iterative}.

\begin{corollary}
\label{cr:minimalenF} Algorithm \ref{al:iterative} finds a minimum
$g$-quasi-matching of $B$ and has time-complexity $O(g(B)|E(G)|)$,
where $g(B)$ is the need of $B$.
\end{corollary}
\begin{proof}
As $B_n=B$, Theorem \ref{th:minimalen_F_i} establishes that $B$ is a
minimum $g$-quasi-matching of $B$. The path $P_i$ can be found using
an augmented Hungarian method: the algorithm performs a breadth-first
search from the vertex $b_\ell$ in such way, that if the vertex whose
neighbors are examined is in $A$, then the search proceeds along its
$F_{i-1}$ incident edges only, but from vertices of $B$, the search
proceeds along the non-$F_{i-1}$-incident edges only. The search tree
$T$ produced in this manner has exchanging levels of non-$F_{i-1}$ and
$F_{i-1}$ edges, and in $T$ there is a unique $F_{i-1}$-augmenting path
from any vertex to $B$. This path starting at a vertex $a\in A$ of
minimum $F_{i-1}$-degree is the path $P_{i-1}$ required for Algorithm
\ref{al:iterative}. The whole tree $T$ (and thus the augmenting path
$P_{i-1}$ can be constructed in $O(|E|)$ time. As there are
$g(B)=\sum_{b\in B}g(b)$ iterations, the overall complexity of Algorithm
\ref{al:iterative} is $O(g(B)|E(G)|)$.
\end{proof}
Note that Theorem \ref{th:declinenov} can be applied to prune the
tree constructed in the generalized Hungarian method in such a way,
that the search tree contains vertices of one $F_{i-1}$-degree only.
If $d$ is the minimum $F_{i-1}$-degree of a neighbor of $b_i$, then
$P_{i-1}$ need not contain any vertex of degree $d+1$. Furthermore,
as soon as a vertex of $F_{i-1}$-degree $d-1$ is encountered, we can
assume that this is the terminating vertex of $P_{i-1}$. These
observations do not improve the theoretical complexity of the algorithm
(in the worst case, for instance when $G$ has a perfect matching, we
still need to consider $O(|E(G)|)$ edges at each iteration), but they
could considerably improve any practical implementation.


\section{On-line application of Algorithm \ref{al:iterative}}
\label{sc:online}

Note that each step of Algorithm \ref{al:iterative} can be viewed as
a part of an on-line procedure, where the need of a vertex, denoted
$b_{\ell}$, increases by one. In particular, this allows for
immediate application of this algorithm to the on-line setting
--- to rearrange it for the on-line addition of a new vertex $v$
with need $g(v)$, one only needs to perform one step of
the outer {\bf while} loop (hence the inner {\bf while} loop which
takes $O(|E(G)|)$ time is performed $g(v)$ times).

However, the full on-line setting, as presented in \cite{azar}, also
allows for removal of the vertices of $B$, i.e. an on-line event is
not just appearance of a new vertex, but also disappearance of an
existing vertex. In our setting, this would correspond to a wireless
sensor malfunction or running out of battery, and in the
task-scheduling setting of \cite{azar}, this corresponds to a task
being removed from the schedule or the number of required machines
for the task being decreased.

Algorithm \ref{al:remove} describes how to augment an existing
minimum quasi-matching when the need of a single vertex $b\in B$
decreases by one to obtain an optimal quasi-matching with respect to
the new need function. As above, if $b$ disappears, then this
algorithm simply needs to be performed $g(v)$ times.

Let $G=A+B$ be a bipartite graph with $B=\lset{b_1}{b_{n}}$, and let
$b\in B$, say $b=b_k$ for some $k$. If $g\colon B\rightarrow \NN$ is
a need function of $B$, then we denote by $g_b$ the mapping from $B$
to $\NN$ with $g_b(b_i)=g(b_i)$ for $i\neq k$, and $g_b(b)=g(b)-1$.

\begin{algorithm}
\caption{Obtaining a minimum $g_b$-quasi-matching from a minimum
$g$-quasi-matching in $G=A+B$.} \label{al:remove}
\begin{algorithmic}
\PARAMETER{$G=A+B$} a bipartite graph with $B=\lset{b_1}{b_{n}}$ and
need function $g\colon B\rightarrow \NN$. \PARAMETER{$F$} a minimum
$g$-quasi-matching of $B$ in $G$. \PARAMETER{$b$} a vertex of $B$.
\OUTPUT{$F'$} a minimum $g_b$-quasi-matching in $G$.
  \STATE Set $A_b$ be the set of $F$-neighbors of $b$.
  \STATE Set $a\in A_b$ be the vertex with largest $F$-degree in $A_b$.
  \IF{there is a backward $F$-alternating path $P$ in $G$ from $a'\in A_b$
  to $a''\in A$ with $d_{F}(a'')=d_{F}(a')+1$}
        \STATE set $F'=F\oplus P-a'b$
    \ELSE
        \STATE set $F'=F-ab$.
    \ENDIF
    \RETURN $F'$.
\end{algorithmic}
\end{algorithm}

\begin{theorem}
\label{th:removal} Let $G=A+B$ be a bipartite graph and $F$ a
minimum $g$-quasi-matching of $B$ in $G$. Using the notation and
assumptions of Algorithm \ref{al:remove}, $F'$ is a minimum
$g_b$-quasi-matching of $B$ in $G$.
\end{theorem}
\begin{proof}
By Theorem \ref{th:declinenov}, we need to prove that every
$F'$-alternating path has decline at most 1 in $G$. Note that every
$F$-alternating path has decline at most 1 in $G$, since $F$ is
minimum by assumption. There are two cases in the algorithm that we
deal with separately.

Suppose first there is no such backward $F$-alternating path $P$ in
$G$ from $a'\in A_b$ to $a''\in A$ with $d_{F}(a'')=d_{F}(a')+1$.
Then $F'=F-ab$, and note that $d_F(A)=d_{F'}(A)$ except in $a$ where
$d_{F'}(a)=d_F(a)-1$. Hence, if there is any $F'$-alternating path
with decline greater than 1, it ends in $a$. Now, no such violating
path could start with a vertex from $A_b$, since $a$ has the largest
$F$-degree among these vertices. And also, no such violating path
could start in any other vertex $a''$ of $A$, because that would mean
there is a backward $F$-alternating path in $G$ from $a\in A_b$ to
$a''\in A$ with $d_{F}(a'')=d_{F}(a')+1$, contrary to our
assumption.

Secondly, suppose there exists a backward $F$-alternating path in
$G$ from $a'\in A_b$ to $a''\in A$ with $d_{F}(a'')=d_{F}(a')+1$,
and let $P$ be a shortest such path. Then $F'=F\oplus P-a'b$, and we
have $d_F(A)=d_{F'}(A)$ except in $a''$ where
$d_{F'}(a'')=d_F(a'')-1$. By the choice of $P$ and the fact that
there are no $F$-alternating paths with decline more than one, we
infer that $d_F(v)=d_F(a')$ for all vertices $v\in A$ on
$P\setminus\{a''\}$. Hence, for all vertices $v\in A$ on $P$ ($a''$
included), we have $d_{F'}(v)=d_{F'}(a')$.  For the purpose of
contradiction let us suppose there is a violating $F'$-alternating
path $P'$ from $\hat{a}$ to $\tilde{a}$. Since $F'$ and $F$ differ
only on $P$, we infer that $P'$ must intersect $P$ in some vertex of
$A$. This readily implies that $d_{F'}(\hat{a})\leq d_{F'}(a')+1$
and $d_{F'}(\tilde{a})\geq d_{F'}(a')-1$. Since $P'$ is violating, we
infer that in fact $d_{F'}(\hat{a})=d_{F'}(a')+1$ and
$d_{F'}(\tilde{a})=d_{F'}(a')-1$ so that the decline of $P'$ with
respect to $F'$ is exactly 2. Now, we can easily find that there is
an $F$-alternating path from ${a''}$ to $\tilde{a}$ in $G$ whose
decline equals 2, which is a contradiction with $F$ being a minimum
$g$-quasi-matching.
\end{proof}

From Theorem \ref{th:removal} and previous discussion, we infer that
the augmented Hungarian method presented in this paper can be
applied to the on-line problem of constructing an optimal
quasi-matching of $B$ with the set $A$ fixed, when the vertices of
$B$ either appear or disappear one at a time. Each on-line step
assures optimality of the current quasi-matching in $O(g(v)|E(G)|)$
steps. Moreover, a similar approach could be used for on-line
setting, where the vertices of $A$ can appear or disappear. When a
vertex of $A$ of $F$-degree $d$ is removed, its $F$-neighbors from
$B$ loose the degree with respect to a quasi-matching, which can be
iteratively recovered, resulting in a patching algorithm of
complexity $O(d|E(G)|)$. On the other hand, when an $A$-vertex of
$G$-degree $d$ is added, up to $d$ vertices can be assigned to it,
again resulting in a $O(d|E(G)|)$ algorithm per on-line step. These
(rather technical) issues are treated in greater detail in a sequel
paper \cite{protocol}, which is oriented towards the mentioned
application.

Note that our adaptation of Hungarian method is, when reduced to
semi-matchings and only addition of $b$-vertices, the same as in
\cite{hllt}. However, our proof of correctness differs in that we
explicitly maintain minimality of the constructed semi-matching (in
fact, even an arbitrary $g$-quasi-matching), after each addition (or
removal) of a vertex. Furthermore, the set of possible alternating
paths with decline at least two is in our approach narrowed to the
vertex that is added to or removed from the graph, resulting in an
efficient on-line version of the algorithm.


\section{Generalized Hall's marriage theorem}
\label{sc:hall}

In this section, we present a solution to Problem \ref{pb:minDegGen}
by characterizing bipartite graphs $A+B$ with given $f\colon A
\rightarrow \NN$ and $g\colon B \rightarrow \NN$ that admit an
$f,g$-quasi-matching. The result is a vast generalization of Hall's
theorem.

A {\em network} $N=(V,A)$ is a digraph with a nonnegative capacity
$c(e)$ on each edge $e$, and with two distinguished vertices: {\em
source} $s$ and {\em sink} $t$ (usually, $s$ has only outgoing, and
$t$ has only ingoing arcs). A {\em flow} $g$ assigns a value $fl(e)$
to each edge $e$. A flow $fl$ is {\em feasible} if for each edge
$e$, $0\leq fl(e)\leq c(e)$ and the {\em conservation (Kirchhoff's)
law} is fulfilled: for every vertex $v\in V(N)\setminus \{s,t\}$,
$$\sum_{vx\in A(N)}{fl(vx)}=\sum_{xv\in A(N)}{fl(xv)}.$$
The {\em value} of a flow $fl$ is $\sum_{sx\in A(N)} fl(sx)$, which
is equal to $\sum_{xt\in A(N)} fl(xt)$. The famous Ford-Fulkerson
(or max-flow min-cut) theorem states that the maximum value of a
feasible flow in $N$ coincides with the minimum capacity of a cut in
$N$. (Where {\em cut} is the set of arcs from $S$ to $T$ in a $S,T$
partition of $N$ (i.e. $s\in S,t\in T$), and its {\em capacity} is
the sum of the $c$-values of its edges). More on this well-known
problem and theorem can be found for instance in
\cite{lp-86,west-01}. One of the several proofs of the famous Hall's
marriage theorem uses the max-flow min-cut theorem, and in our
generalization of Hall's theorem, we will follow similar lines.

\begin{definition}
Let $G=A+B$ be a bipartite graph, $f\colon A\rightarrow \NN$ an
availability function, and $Y\subseteq B$. For $x\in A$, let
$d_Y(x)=|\{y\in Y\,:\,xy\in E(G)\}|$, that is the number of
neighbors of $x$ from $Y$. For $X\subset A$, let $f(X,Y)=\sum_{x\in
X}{\min\{f(x),d_Y(x)\}}$ denote the {\em relative availability} of
$X$ with respect to $f$ and $Y$. In particular, for $x\in X$, we
write $f(\{x\},Y)$ as $f(x,Y)$ (which is the least of $f(x)$ and
$d_Y(x)$).
\end{definition}

Intuitively, the relative availability of $X$ with respect to $f$
and $Y$ presents the maximum number of edges going from $X$ that can
be used to cover $Y$.

\begin{theorem}
\label{THall} Let $G=A+B$ be a bipartite graph, with
$A=\lset{a_1}{a_m}$, $B=\lset{b_1}{b_n}$,  a mapping $f\colon A
\rightarrow \NN$, and $g\colon B \rightarrow \NN$. Then $G$ has an
$f,g$-quasi-matching of $A+B$ if and only if for every $Y\subseteq
B$,
\begin{equation}
\label{Ehall}
f(N(Y),Y)\geq g(Y).
\end{equation}
\end{theorem}
\begin{proof}
Suppose there is a subset $Y\subset B$ such that $\sum_{u\in
N(Y)}{f(u,Y)}=f(N(Y),Y)< g(Y)=\sum_{v\in Y}{g(v)}$. Let $F$ be an
arbitrary $g$-quasi-matching of $B$ in $G$. The vertices of $Y$
altogether must have at least $g(Y)$ $F$-neighbors. As the relative
availability of their neighbors $N(Y)$ is less than $g(Y)$, we
derive by the pigeon-hole principle that there will be a vertex
$u\in N(Y)$ such that $d_F(u)>f(u)$. Hence $F$ is not an
$f,g$-quasi-matching, which readily implies (since $F$ was
arbitrarily chosen) that no $f,g$-quasi-matching exists.

For the converse, let $f(N(Y),Y)\geq g(Y)$ hold for all $Y\subseteq
B$. We introduce two additional vertices: $a$ that is connected to
all vertices $a_i\in A$, and $b$, connected to all $b_j\in B$.
Construct a digraph $G'$, by choosing a direction of all edges from
$G$ as follows: from $a$ to each $a_i\in A$, from vertices of $A$ to
their neighbors in $B$, and from each $b_j$ to $b$. Next, construct
a network out of the digraph $G'$, by setting flow capacities
$c:E(G')\rightarrow \NN$  as follows: $c(aa_i)=f(a_i)$,
$c(a_ib_j)=1$ (for $a_ib_j\in E(G)$), and $c(b_jb)=g(b_j)$. Note
that there exists a flow of size $g(B)$ in $G'$ if and only if there
exists an $f,g$-quasi-matching of $A+B$. By max-flow min-cut
theorem, the maximum flow value coincides with the minimum cut
capacity in the network $G'$.

Let $C$ be a minimum cut in the network, and let $Z$
be the set of vertices from $B$ for which $b_jb\in C$.
Let $Y=B\setminus Z$. Since $C$ is a cut, for every
vertex $b_j\in Y$ and every neighbor $a_i$ of $b_j$, we
have either $a_ib_j\in C$ or $aa_i\in C$ (since $C$ is
minimum, we may assume that both does not happen).
Denote by $K$ the set of vertices $a_i$ from $N(Y)$
such that $aa_i\in C$ and let $L=N(Y)\setminus K$.
For $b_j\in Y$, let $m_j$ denote the number of its
neighbors in $L$ (which coincides with the number of its
incident edges that are from $C$). Note that
$$\sum_{j,b_j\in Y} {m_j}=\sum_{a_i\in L}{d_Y(a_i)}\geq f(L,Y).$$
Now,
\begin{eqnarray*}
|C| & = & g(Z)+f(K)+\sum_{j,b_j\in Y}{m_j}\\
    & \geq & g(Z)+f(K,Y)+f(L,Y)\\
    & \geq & g(Z)+f(N(Y),Y) \\
    & \geq & g(Z)+g(Y)=g(B) \\
\end{eqnarray*}
where in the last inequality \eqref{Ehall} is used. 
The result now readily follows.
\end{proof}

The theorem has several corollaries. We state the most obvious.
First, if $f$ is not involved, i.e. if $f(u)=d(u)$ for all $u\in A$,
then $f(N(Y),Y)=\sum_{u\in N(Y)}{d_Y(u)}=\sum_{v\in Y}{d(v)}$, and
\eqref{Ehall} turns into a much simpler condition 
$\sum_{v\in Y}{d(v)}\geq g(Y)$ for
every $Y\subseteq B$.

If we want that each vertex in $A$ covers only one vertex from $B$,
that is $f(u)=1$ for all $u\in A$, we get $f(N(Y),Y)=\sum_{u\in
N(Y)}{1}=|N(Y)|$, and the condition \eqref{Ehall} reads $|N(Y)|\geq
g(Y)$ for every $Y\subseteq B$. If, in addition, $g(v)=1$ for all
$v\in B$, we get $|N(Y)|\geq |Y|$ for all $Y\subseteq B$ which is
exactly Hall's condition. On the other hand, this implies that $A+B$
has a perfect matching of vertices from $B$. Thus Hall's theorem is
a corollary of Theorem \ref{THall}.

One of the common formulations of Hall's theorem is in terms of
systems of distinct representatives. Let us formulate also Theorem
\ref{THall} in this sense.

Let ${\cal A}=\{A_1,\ldots, A_m\}$ be a family of sets, with
$S=\cup_{i=1}^m{A_i}=\{b_1,\ldots,b_n\}$, and let there be mappings
$f\colon {\cal A} \rightarrow \NN$, and $g\colon S \rightarrow \NN$.
We say that the family {\cal A} has a {\em (lower) system of
$f,g$-representatives} if to every set $A_i\in \cal A$ we associate
at most $f(A_i)$ representatives from $S$, and every vertex $b_j\in
S$ is a representative of at least $g(b_j)$ sets from $\cal A$. In
this terminology, Theorem \ref{THall} reads as follows.

\begin{corollary}
A family of sets $\cal A$ has a lower system of
$f,g$-representatives if and only if for every subset $Y\subseteq S$
we have
$$\sum_{A_i\in {\cal A}} {\min\{f(A_i),|A_i\cap Y|\}}\geq \sum_{b_j\in Y}{g(b_j)}.$$

\end{corollary}

By duality, since the interpretation of the roles of sets and
vertices in Theorem \ref{THall} can be reversed, we have another
corollary expressed in similar terms. Let ${\cal B}=\{B_1,\ldots,
B_n\}$ be a family of sets, with
$S=\cup_{j=1}^n{B_j}=\{a_1,\ldots,a_m\}$, and let there be mappings
$f\colon {S} \rightarrow \NN$, and $g\colon {\cal B} \rightarrow
\NN$. We say that the family ${\cal B}$ has {\em an upper system of
$f,g$-representatives} if to every set $B_j\in \cal B$, we associate
at least $g(B_j)$ representatives from $S$, and every vertex $a_i\in
S$ is a representative of at most $f(a_i)$ sets from $\cal B$. In
this terminology, we infer from Theorem \ref{THall}:

\begin{corollary}
A family of sets $\cal B$ has an upper system of
$f,g$-representatives if and only if for every subfamily $Y\subseteq
\cal B$ we have
$$\sum_{a_i\in S} {\min\{f(a_i),|Y(a_i)|\}}\geq \sum_{B_j\in Y}{g(B_j)},$$
\noindent where $Y(a_i)=\{B_j\in Y \,\colon\, a_i\in B_j\}$ (i.e.
$|Y(a_i)|$ is the number of sets from the family $Y$ that contain
$a_i$).
\end{corollary}

From the above corollaries, one can easily find formulations when one
or both of the mappings $f,g$ is not involved or is constant (say,
equal to 1). The resulting formulations are mostly easier and nicer
as the above and could also be applicable.

\section*{Acknowledgement}

We thank to Matja\v z Kov\v se for fruitful discussions during the preparation of this paper.

\end{document}